\documentclass{amsart}
\usepackage{graphicx,amsfonts,amsmath,amssymb,amsthm,abstract,url,hyperref,mathrsfs} 
\usepackage[margin=1in]{geometry}
\usepackage[colorinlistoftodos]{todonotes}

\usepackage[style=alphabetic]{biblatex} 
\bibliography{refs}
\usepackage[all]{xy}

\theoremstyle{plain}
\newtheorem{thm}{Theorem}
\newtheorem{lemma}[thm]{Lemma}

\newtheorem{conj}[thm]{Conjecture}

\theoremstyle{definition}
\newtheorem{defn}[thm]{Definition}

\theoremstyle{remark}

\newtheorem{ex}[thm]{Example}

\newtheorem*{notn}{Notation}

\numberwithin{equation}{section}
\numberwithin{thm}{section}

\DeclareMathOperator{\cl}{cl}
\DeclareMathOperator{\ac}{ac}
\DeclareMathOperator{\aw}{aw}
\DeclareMathOperator{\mex}{mex}

\title{Evil Twins in Sums of Wildflowers}
\author{Simon Rubinstein-Salzedo}
\address{Euler Circle, Mountain View, CA 94040}
\email{simon@eulercircle.com}
\author{Stephen Zhou}
\email{stephenzh99@gmail.com}
\date{\today}

\begin{document}

\maketitle
\begin{abstract}
     A game $G$ is said to have the evil twin property if there exists $G^* \in \{G,G+*\}$ such that $o^+(G) = o^-(G^*)$ and $o^+(G^*) = o^-(G)$. We study sums of wildflowers, games of form $G:H$. We find that a large closed set of sums of wildflowers has the evil twin property, extending work done in \cite{mckay2016misere} and \cite{lomisere}. Our argument partially generalizes the mis\`ere genus theory of Conway to partizan games, and requires proving several general theorems on ways to extend sets with the evil twin property. Many sums of mutant flowers of the form $\{*x_1,\dots,*x_n\}:a$, where $a$ is a number, also have the evil twin property. We also prove that this set of mutant flowers is the largest such closed set with the evil twin property, and that it is $\mathsf{NP\text{-}hard}$ to compute the outcome class of a sum of mutant flowers under either play convention via a reduction from \textsc{3-Sat}. Previous work on this topic was done by McKay, Milley, and Nowakowski, and later Lo. 
\end{abstract}
\section{Introduction}
\label{2}

Games are much easier to understand under the normal play convention than the mis\`ere play convention. The crucial difference is that games form a group under addition in normal play, but only a monoid under mis\`ere play. However, some research still has been done on mis\`ere play games.

\begin{notn}
    Let $G$ be a combinatorial game. The set of Left options of $G$ are denoted by $G^{\mathcal{L}}$ and the set of Right options by $G^{\mathcal{R}}$, while specific Left and Right options will be denoted by $G^L$ and $G^R$, or a variant thereof.
\end{notn}

The games we study in the paper can be naturally expressed in terms of \emph{ordinal sums}.

\begin{defn}
    $G:H = \{G^\mathcal{L},G:H^\mathcal{L}\mid G^\mathcal{R},G:H^\mathcal{R}\}$ is the \emph{ordinal sum} of two games $G$ and $H$. 
\end{defn}
Notice that ordinal sums depend on the form of a game; that is, there exist $G_1 = G_2$ such that $G_1:H \neq G_2:H$. However, in normal play, it is true that $G:H_1 = G:H_2$ if $H_1 = H_2$. The ordinal sum is also associative, so that $(G:H):J = G:(H:J)$ for all games $G,H,J$.

\begin{defn}
$\mathbb{D}$ is the set of all \emph{dyadic rationals} of form $\frac{x}{2^y}$, where $x$,$y$ are integers. \end{defn}

\begin{defn}
A \emph{sprig} is a game of form $*:a$, where $a$ is a dyadic rational.
\end{defn}
The following symmetry between the outcome classes of normal and mis\`ere play games is proven in \cite{mckay2016misere}. 
\begin{thm}
    \label{mckay_thm}
    Let $F = \sum_{i=1}^m *:a_i$ be a sum of sprigs. Define $F^* = F+*$. Then $o^+(F) = o^-(F^*)$ and $o^-(F^*) =o^+(F)$.
\end{thm}
This property is interesting because it allows us to determine the mis\`ere play outcome from the normal outcome. Lo extends this theorem in \cite{lomisere}.
\begin{defn}
    A (generalized) \emph{flower} is a game of form $*n:a$, where $a$ is a dyadic rational.
\end{defn}
\begin{thm}
    \label{lo_thm}
    Let $F = \sum_{i=1}^n *n_i:a_i$ be a sum of flowers. Let $m = \max \{n_i\}$. Define \begin{equation}
        F^* = \begin{cases}
            F & \text{if }m>1 \\
            F+* & \text{if } m = 1.
        \end{cases}
    \end{equation}
    Then $o^+(F) = o^-(F^*)$ and $o^+(F^*) = o^-(F)$.
\end{thm}

A similar fact is true for all impartial mis\`ere games. 

\begin{thm}
    \label{impartial_evil}
    Let $G$ be a tame impartial game. Then there exists $G^* \in \{G,G+*\}$ such that $o^+(G) = o^-(G^*)$ and $o^+(G^*) = o^-(G)$.
\end{thm}

\begin{proof}
Let $\mathscr{T}$ be the quotient of all tame games. It is known that for all $G \in \mathscr{T}$, there exists a sum of nimbers $H$ such that $G \equiv_{\mathscr{T}} H$ and $G = H$. In particular, $o^-(G) = o^-(H)$, $o^-(G+*) = o^-(H+*)$ and $o^+(G) = o^+(H)$, $o^+(G+*) = o^+(H+*)$. Since either $o^+(H) = o^-(H)$ or $o^+(H+*) = o^-(H)$ and $o^-(H+*) = o^+(H)$, the same is true for $G$. 
\end{proof}

The purpose of this paper is to extend Theorem \ref{lo_thm} to a larger set of games, including many mutant flowers of the form $\{*x_1,*x_2,\dots,*x_m\}:a$, where $a$ is a dyadic rational. The terminology of the following theorem will be defined in Sections \ref{3} and \ref{wildflowers}.  \begin{thm}
  \label{main theorem 1, minus teminology}
    Let $\mathscr{S}$ be the set of all fickle restricted wildflowers, and let $\mathscr{H}$ be the set of all firm restricted wildflowers. Define
    \begin{equation}
        G^* = \begin{cases}
            G & \text{if } G \notin \cl(\mathscr{L}) \\
            G+* & \text{if } G \in \cl(\mathscr{L}).
        \end{cases}
    \end{equation}Then $o^+(G) = o^-(G^*)$ and $o^-(G) = o^+(G^*)$. 
\end{thm}

\begin{thm}
 \label{main theorem for mutant flowers, minus terminology}
    Let $\mathscr{M}$ be the set of all sums of mutant flowers. Say that $G \in\mathscr{M} \cap \cl(\mathscr{S}\cup \mathscr{H}) $, and let the maximum height of a flower in $G$ be $m$. Define \begin{equation}
        G^* = \begin{cases}
            G & \text{if } m > 1 \\
            G+* & \text{if } m\leq 1.
        \end{cases}
    \end{equation}
    Then $o^+(G) = o^-(G^*)$ and $o^-(G) = o^+(G^*)$. 
\end{thm}

In Section \ref{3}, we define terminology that will be used in later sections and present some conjectures on the structure of sets with the evil twin property. In Section \ref{4}, we prove some general theorems on sets with the evil twin property. In Section \ref{wildflowers}, we define the types of games we will investigate and state our main theorem. In Section \ref{5}, we prove some results specific to wildflowers that enable us to prove our main theorem. In Section \ref{6}, we prove that it is difficult to find the outcome class of a sum of wildflowers despite the main theorem. We conclude with some questions in Section \ref{7}.

Before we start, we must make our notation clear. There are several definitions of equality, inequality,  and outcome classes that are easily confused, so we will be careful to first make our notation clear before proving results.

Any combinatorial game $G$ has four outcome classes, denoted $\mathcal{L}$, for Left winning moving both first and second, $\mathcal{R}$ for Right winning moving both first and second, $\mathcal{N}$ for the first player always winning, and $\mathcal{P}$ for the second player always winning. These option classes are partially ordered by favorability for Left  as shown in Figure~\ref{fig:outcomeposet}. For example, $\mathcal{L} > \mathcal{P}$, but $\mathcal{N}$ cannot be compared to $\mathcal{P}$.

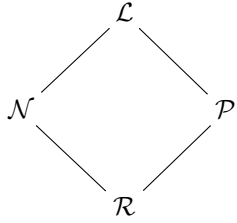
\begin{figure}
\[\xymatrix{ & \mathcal{L}\ar@{-}[dl]\ar@{-}[dr] \\ \mathcal{N}\ar@{-}[dr] && \mathcal{P}\ar@{-}[dl] \\ & \mathcal{R}}\]
\caption{The poset of outcome classes.}
\label{fig:outcomeposet}
\end{figure}

These outcome classes can differ between play conventions. The outcome class of a game $G$ under normal play will be denoted $o^+(G)$, and the outcome class of $G$ under mis\`ere play will be denoted $o^-(G)$. For example, $o^+(0) = \mathcal{P}$, while $o^-(0) = \mathcal{N}$. We define $\mathcal{O}^+$, where $\mathcal{O} \in \{\mathcal{L,R,N,P}\}$, to be the set of all games $G$ where $o^+(G) = \mathcal{O}$, and $\mathcal{O}^-$ to be the set of all games $G$ where $o^-(G) = \mathcal{O}$. This means that $G \in \mathcal{O}^+$ is equivalent to $o^+(G) = \mathcal{O}$, and similarly for $\mathcal{O}^-$. 

The simplest way to define equality is in terms of game trees. If $G$ and $H$ have the same game tree, we say that they are literally equal or identical, and write $G \cong H$.

We write $G = H$ if $o^+(G+X) = o^+(H+X)$ for all games $X$, and $G \geq H$ if $o^+(G+X) \geq o^+(H+X)$ for all games $X$. There are easy tests for equality and inequality in normal play. Define the conjugate of a game $G = \{\mathcal{G}^{\mathcal{L}}\mid\mathcal{G}^{\mathcal{R}}\}$ recursively as $\overline{G }= \{\overline{\mathcal{G}^{\mathcal{L}}}\mid \overline{\mathcal{G}^{\mathcal{R}} }\}$. In normal play, $G + \overline{G} = 0$, so the conjugate is usually written as the negative $-G = \overline{G}$. 
It can be shown that $G = H$ if and only if $G-H \in \mathcal{P}^+$, and $G\geq H$ if and only if $G - H \in \mathcal{L}^+ \cup \mathcal{P}^+$. 

Inequality in mis\`ere play is defined analogously. We write $G=^- H$ if $o^-(G+X) =o^-(G+H) $ for all games $X$ and $G \geq^- H$ if $o^-(G+X) \geq o^-(H+X)$ for all games $X$. However, there are no easy tests for equality or inequality in mis\`ere play. For example, Ottoway showed in \cite{ottoway_zero} that $G+\overline{G} \neq^- 0$ unless $G$ is identical to $0$, which is the reason the notation $-G$ is avoided in mis\`ere play.

The difficulty of finding games that are equal prevented much progress from being made in mis\`ere play until Plambeck and Siegel introduced \emph{mis\`ere quotients} in \cite{plambeck2008misere}. Mis\`ere quotients restrict the games $X$ that can be added to $G$ and $H$.
A set of games $\mathscr{A}$ is \emph{closed under addition} if for any $G,H \in \mathscr{A}$, we have $G +H \in \mathscr{A}$. A set $\mathscr{A}$ is said to be \emph{hereditarily closed} if for any $G \in \mathscr{A}$ and any subposition $G'$ of $G$, we have $G' \in \mathscr{A}$. A set $\mathscr{A}$ is said to be \emph{closed} if $\mathscr{A}$ is both closed under addition and hereditarily closed. The closure $\cl (\mathscr{A})$ of a set $\mathscr{A}$ of games is the smallest closed set of games containing $\mathscr{A}$. The additive closure $\text{ac}(\mathscr{A})$ of a set $\mathscr{A}$ is the smallest set of games closed under addition containing $\mathscr{A}$.

Then for $G,H \in \mathscr{A}$, with $\mathscr{A}$ a closed set, we write $G \equiv_{\mathscr{A}} H$ if $o^-(G+X)=o^-(H+X)$ for all $X \in \mathscr{A}$, and $G \geq_{\mathscr{A}} H$ if $o^-(G+X)\geq o^-(H+X)$ for all $X \in \mathscr{A}$. 

A particularly important quotient is that of dicotic games, those where either both Left and Right have options, or neither do. The set of all dicotic games is denoted $\mathscr{D}$. An important equality in the universe of dicotic games is $*+* \equiv 0 \pmod {\mathscr{D}}$ which was proven in \cite{allen2015peeking}.

We also use the mis\`ere genus theory of Conway. Let $G$ be an impartial game. We define the normal play Grundy value of $G$ recursively as \begin{equation}
    \mathcal{G}^+(G) = \begin{cases}
    \mex(G') & G \not \cong 0 \\
    0 & G  =0\\
\end{cases}.
\end{equation} The mis\`ere play Grundy value of $G$ can be found similarly as \begin{equation}
    \mathcal{G}^-(G) = \begin{cases}
    \mex(G') & G \not \cong 0 \\
    1 & G  =0\\
\end{cases}.
\end{equation} The genus of $G$ is denoted $a^b$, where $a = \mathcal{G}^+(G)$ and $b = \mathcal{G}^-(G)$. It is known that $\mathcal{G}^+(G)$ is the only $n$ such that $G+*n \in \mathcal{P}^+$, and symmetrically for $\mathcal{G}^-(G)$. A game $G$ is said to be \emph{tame} if the genus of $G$ is $1^0$ ,$0^1$ or $n^n$. A tame game is \emph{fickle} if it's genus is $1^0$ or $0^1$, and \emph{firm} otherwise. It is known that fickle games act like $0$'s and $*$'s in a sum of nimbers under mis\`ere play, while firm games act like $*n + *2+*2$.  An application of genus theory is to solve certain mis\`ere play octal games. For more on genus theory, see \cite{conway2000numbers}. We will use genus theory to define certain sets with desirable properties.

\section{Evil Twins and Kernels}
\label{3}

We give a name to the property of Theorem \ref{lo_thm}.

\begin{defn}
    A closed set of games $\mathscr{A}$ with $* \in \mathscr{A}$ is said to have the evil twin property if there exists a subset $\mathscr{K}  \subseteq \mathscr{A}$ such that if we define \begin{equation}
          G^* = \begin{cases}
            G & \text{if } G \in \mathscr{K} \\
            G+* & \text{if } G \notin \mathscr{K},
            \end{cases}
    \end{equation}
    then $o^+(G) = o^-(G^*)$ and $o^-(G) = o^+(G^*)$. The subset $\mathscr{K}$ is called an \emph{evil kernel} of $\mathscr{A}$. 
\end{defn}

\begin{ex}
    Theorem \ref{mckay_thm} is equivalent to the statement that $\text{cl}(\{*:a \mid a \in \mathbb{D}\})$ has the evil twin property with evil kernel $\varnothing$. 
\end{ex}

\begin{ex}
    \label{lo_example}
   Theorem \ref{lo_thm} is equivalent to the statement that  $\cl (\{*n:a \mid a \in \mathbb{D},n \in \mathbb{N}\})$ has the evil twin property with evil kernel $\cl (\{*n:a \mid a \in \mathbb{D},n \in \mathbb{N}\})\setminus \cl (\{*:a \mid a \in \mathbb{D}\})$.
\end{ex}
\begin{ex}
    Theorem \ref{impartial_evil} is equivalent to the statement that the set of tame games has the evil twin property, where the evil kernel is the set of all fickle games.
\end{ex}
Notice that if an impartial game $G$ is in the kernel, then $G$ must be firm, and vice versa if $G$ is outside the kernel. Sets with the evil twin property partially generalize Conway's genus theory to partizan games. $\mathscr{K}$ is called \emph{an} evil kernel because $\mathscr{K}$ is not necessarily unique. For example, we have that $*2:1 + *2:1 \in \mathcal{L}^+ \cap \mathcal{L}^-$, and $*2:1 + *2:1+* \in \mathcal{L}^+ \cap \mathcal{L}^-$, so it is irrelevant if $*2:1+*2:1 \in \mathscr{K}$ or not. Thus, $\cl (\{*n:a \mid a \in \mathbb{D},n \in \mathbb{N}\})\setminus (\cl (\{*:a \mid a \in \mathbb{D}\}) \cup G)$, where $G = *2:1+*2:1$, is also an evil kernel.

 $\mathscr{K}$ is called a kernel due to an analogy to the kernels of impartial mis\`ere quotients, which are defined in \cite{plambeck2008misere}. Plambeck and Siegel prove the following theorem in \cite{plambeck2008misere}.
 \begin{thm}
     Let $\mathscr{K}$ be the kernel of a finite faithful regular mis\`ere quotient $\mathcal{Q} (\mathscr{A})$. If $G \in \mathscr{A}$, then a winning move in $G$ under normal play is also a winning move under mis\`ere play, unless that move is to a game outside of the kernel.
\end{thm}
The analogous theorem for games with the evil twin property is the following.
\begin{thm}
    Let $\mathscr{K}$ be the kernel of a set $\mathscr{A}$ of games with the evil twin property. If $G \in \mathscr{A}$, then a winning move in $G$ under normal play is also a winning move under mis\`ere play, unless that move is outside of the kernel.
\end{thm}
\begin{proof}
    Let the winning move in normal play be to move to $G' \in \mathscr{K}$. Then since $o^+(G') = o^-(G')$, $G'$ is a winning move in mis\`ere play as well.
\end{proof}

\section{Normal evilness and star-closure}
\label{4}
In this section, we prove two theorems allowing us to extend sets with the evil twin property, assuming that there exists an evil kernel with a certain structure. 
We will first give a short proof of Theorem \ref{lo_thm} assuming Theorem \ref{mckay_thm} to demonstrate some of the methods used.

\begin{proof}[Proof of Theorem \ref{lo_thm}]
     Let $\mathscr{A} = \text{cl}(*n:a), n > 0$ be the set of all sums of generalized flowers, and $\mathscr{K} = \text{cl}(*:a), n > 1$. Define $G^* = G$ for $G \in \mathscr{K}$ and $G^* = G+*$ for $G \in \mathscr{A}\setminus \mathscr{K}$. Then we need to prove that $o^-(G) = o^+(G^*)$ and $o^-(G^*)=o^+(G)$.
    
We will induct on the birthday of $G$, so we can assume that the theorem has already been proven for the options of $G$.  Since  Theorem \ref{mckay_thm} covers the case where all the flowers are of height $1$, so we can assume $G \in \mathscr{K}$, or that $G$ has a flower of height $>1$. We will show that for every Left option $G^L$ of $G$, there exists another option ${G^{L}}'$ such that $o^+(G^L) = o^-({G^L}')$ and $o^-(G^L) = o^+({G^L}')$. This would imply that $o^-(G) = o^+(G)$, since there exist options with the same outcome class in both play conventions.

If $G^{L} \in \mathscr{K}$, then ${G^{L}}'= G^L$ works. Now suppose $G^{L} \in \mathscr{A}\setminus\mathscr{K}$. Then Left must have moved to $G^L$ by moving in the last generalized flower of height $>1$ to either $*$ or $0$. So we can find another Left option ${G^{L}}'$ with the property that $  {G^L}'=G^{L}+* $ in normal play by choosing to move to the other element of $\{0,*\}$ in that last tall flower not chosen in the move from $G$ to $G^L$. Then we calculate that $o^-(G^L) = o^+({G^L}^*) = o^+ ({G^L}')$ and $o^+(G^L) = o^+({G^L}' +*) = o^- ({G^L}' +*+*) = o^-({G^L}')$. Thus, $o^+(G) = o^-(G)$.
 \end{proof}
There were two properties of generalized flowers that made the last proof possible. 
\begin{enumerate}
    \item For all $G,H  \in \mathscr{A}$, we have $G+H \notin \mathscr{K}$ if and only if $G,H \notin \mathscr{K}$.
    \item For all $G \in \mathscr{A}$, if $0$ is a Left (Right) option of $G$, then $*$ is also a Left (Right) option of $G$.
 \end{enumerate}
We will generalize both of these conditions.

The first condition generalizes exactly.
\begin{defn}
    A pair $(\mathscr{A},\mathscr{K})$ of  a closed set $\mathscr{A} \subseteq \mathscr{D}$ and $ \mathscr{K} \subseteq \mathscr{A}$, where $\mathscr{K}$ is an evil kernel, is called \emph{evilly normal} if, for all $A,B \in \mathscr{A}$, $A+B \notin \mathscr{K}$ if and only if $A \not\in \mathscr{K} $ and $B \notin \mathscr{K}$.
\end{defn}
All our previous examples are evilly normal.
\begin{ex}
Theorem    \ref{mckay_thm} is equivalent to the statement that $(\cl(\{(*:a)\mid a \in \mathbb{D}),\varnothing)$ is evilly normal.
\end{ex}
\begin{ex}
  Theorem  \ref{lo_thm} is equivalent to the statement that $(\cl(\{(*n:a)\mid  a \in \mathbb{D},n \in \mathbb{N}\}),\cl(\{(*n:a)\mid a \in \mathbb{D},n \in \mathbb{N}\})\setminus\cl(\{(*:a)\mid a \in \mathbb{D}\}))$ is evilly normal.\end{ex}
\begin{ex}
Let the set of firm games be $\mathscr{H}$ and the set of fickle games be $\mathscr{S}$. Then $(\mathscr{S}\cup \mathscr{H}, \mathscr{H})$  is evilly normal.
\end{ex}
\begin{ex}
  \label{star based counterexample}
    It can be calculated that $\cl (\{0,* \mid *\} , \{*\mid 0,*\}, *2: 1 , *2:-1)$ has $\cl (\{0,* \mid *\} , \{*\mid 0,*\}, *2:1, *2:-1,) \setminus \cl (\{0,* \mid *\} , \{*\mid 0,*\})$ as a kernel.
\end{ex}

We speculate the following:

\begin{conj}
\label{all_norm}
If $\mathscr{A}$ has the evil twin property, then there exists a kernel $\mathscr{K}$ such that $(\mathscr{A},\mathscr{K})$ is evilly normal.
\end{conj}

$(\mathscr{A},\mathscr{K})$ being evilly normal gives information about the $\equiv _{\mathscr{A}}$ relation.
\begin{thm}
    Let $G=H$ in normal play. Then if $G $ and $H$ are both firm, or both fickle, then $G \equiv_{\mathscr{A}} H$.
\end{thm}
\begin{proof}
    Suppose $G,H \in \mathscr{K}$. Then for any $X \in \mathscr{A}$ we have $G+X,H+X \in \mathscr{K}$, so $o^-(G+X) = o^+(G+X) = o^+(H+X) = o^-(H+X)$, where the first and third equalities follow from the definition of $\mathscr{K}$, and second equality follows from the definition of normal play equality.

    Suppose $G,H \in \mathscr{A}\setminus \mathscr{K}$. For any $X \in \mathscr{A}$, either $X \in \mathscr{K}$ or $X \in \mathscr{A} \setminus\mathscr{K}  $. If $X \in \mathscr{K}$, then $G+X \in \mathscr{K},H+X \in \mathscr{K}$, and the proof continues as before. If $X \in \mathscr{A} \setminus \mathscr{K}$, then $G+X,H+X \in \mathscr{A}\setminus \mathscr{K}$, so $o^- (G+X) = o^+(G+X+*) = o^+(H+X+*) = o^- (H+ X)$. 

    Thus $G \equiv_{\mathscr{A}} H$ in all cases.
\end{proof}

Applying the above theorem in the case $G \equiv *$ gives the following:

\begin{thm}
    Let $H =*$ in normal play. Then if $H \in \mathscr{A}\setminus \mathscr{K}$, then $H \equiv_{\mathscr{A}} *$ and for all $G \in \mathscr{A} \setminus \mathscr{K} $, we have $o^-(G+H) = o^+(G)$ and $o^+(G+H) = o^-(G)$.
\end{thm}

Now we generalize the second condition.
\begin{defn}
    A set of games $\mathscr{A}$ is \emph{star-closed} if for any $G \in \mathscr{A}$, there exists $H \in \mathscr{A}$ such that $G +*= H$ in normal play.
\end{defn}
 It is important to remember that the $=$ sign refers to equality in normal play. For example, the set $\{*4,*5\}$ is star-closed, because $*4+* = *5$ and $*5+* = *4$ in normal play, even though neither of these equalities are true in mis\`ere play. 

 We are now ready to state the generalization of the proof of Theorem \ref{lo_thm}.

\begin{thm}
\label{a-k extension}
Let $(\mathscr{A},\mathscr{K})$ be evilly normal. Let $\mathscr{B}$ be an additively closed set of dicotic games such 
that $\mathscr{A} \cup \mathscr{B}$ is hereditarily closed. Suppose that for any $G \in \mathscr{B} \setminus \mathscr{A}$, $G^\mathcal{R} \cap (\mathscr{A}\setminus\mathscr{K}) $ and $G^\mathcal{L} \cap (\mathscr{A}\setminus\mathscr{K}) $ are star-closed. Then $(\cl(\mathscr{A}\cup \mathscr{B}),(\cl(\mathscr{A}\cup \mathscr{B})) \setminus (\mathscr{A} \setminus \mathscr{K}))$ is evilly normal.
\end{thm}
Notice that the complement of the evil kernel remains the same after applying Theorem \ref{a-k extension}.

\begin{proof}[Proof of Theorem \ref{a-k extension}]
     All elements of $\text{cl}(\mathscr{A}\cup \mathscr{B})$ have the form $G \cong A + B$, where $A \in \mathscr{A}$ and $B \in\mathscr{B} \cup \{0\}$ We want to show that $o^-(G) = o^+(G)$ unless $A \in \mathscr{A} \setminus \mathscr{K}$ and $B \cong 0$. We induct on the birthday of $G$. The base case, $A+B \cong 0$, is trivial. We already understand the $B \cong 0$ case, so we can assume that $B \not \cong 0$. We will prove that a winning move in $A+B$ under normal play corresponds to a winning move in mis\`ere play, and vice versa. 
     
     Up to symmetry, we only need to show this for Left. Say that $G \cong A + B$. The Left options of $G$ are of the form $A^L +B$ or $A+B^L$. By induction, since $B \not \cong 0$, $o^-(A^L+B) = o^+(A^L+B)$. So, if Left wins by moving to $A^L+B$ under either play convention, she wins in both. Now consider the case where Left moves to $A+B^L$. If $B^L \in \mathscr{B} \cup  \mathscr{K}$, $o^-(A+B^L)=o^+(A+B^L)$, and so Left wins under either play convention. If $B^L \in \mathscr{A} \setminus \mathscr{K}$, then there exists another Left option of $B$, say ${B^L} '$, such that ${B^L}' = B^L+*$ in normal play. Then we calculate $o^+({B^L}') = o^+(B^L + *) = o^-(B^L+*+*) = o^-(B^L)$ and $o^-({B^L}') =o^+({B^L}' + *) = o^+(B^L+*+*)=o^+(B^L) $. So $B^L$ being a winning move under one convention implies that ${B^L}'$ is a winning move under the other convention. Thus, $o^+(A+B) = o^-(A+B)$ whenever $B \not \cong 0$. 
\end{proof}

Theorem \ref{a-k extension} only allows adding elements to $\mathscr{K}$. The next theorem is for adding elements to $\mathscr{A}\setminus\mathscr{K}$.

\begin{thm}
\label{add to a- (a-k)}
   Let $(\mathscr{A},\mathscr{K})$ be dicotic and evilly normal with the property that $A^\mathcal{L} \cap (\mathscr{A}\setminus\mathscr{K})$ and $A^\mathcal{R} \cap (\mathscr{A}\setminus\mathscr{K})$ are star-closed for any $A \in \mathscr{K}$ and that $A^{\mathcal{L}}\cap \mathscr{K}$ and $A^{\mathcal{R}}\cap \mathscr{K}$ are star-closed for any $A \in (\mathscr{A} \setminus \mathscr{K})$.
    
Let $\mathscr{B}$ be a set of games such that $\mathscr{A} \cup \mathscr{B}$ is hereditarily closed. Assume that for all $B \in \mathscr{B}$, $B^\mathcal{L{}} \cap \cl(\mathscr{K} \cup \mathscr{B})$ and $B^\mathcal{L{}} \cap \cl(\mathscr{K} \cup \mathscr{B})$ are star-closed.

Furthermore, assume that if either player wins by going first in $G+*$ to $G$ for any $G\in \ac( (\mathscr{A}\setminus \mathscr{K} )\cup \mathscr{B})) $ under either play convention, there exists another winning move for that player in $G+*$ in the same play convention.
   
   Then $(\cl(\mathscr{A}\cup \mathscr{B}), \ac(\mathscr{K{ \cup \mathscr{B}}}))$ is evilly normal.
\end{thm}
\begin{proof}
    The elements of $\cl(\mathscr{A} \cup \mathscr{B})$ have form $G= A+B$, where $A \in \mathscr{A}$ and $B \in \mathscr{B}$. We will induct on the birthday of the normal-play canonical form of $G$. This is stronger than induction on the literal form of $G$ since we can assume the theorem is true for any game equal in normal play to an option of $G$.

    Say that $A \in \mathscr{K}$. Then we need to show that $o^+(G) = o^-(G)$. Assume that a player, say Left, wins $G$ going first in normal play. We will show that Left also wins $G$ going first under mis\`ere play.

Say that Left wins $G$ going first under normal play by moving to $G^L$, where $G^L \in \ac(\mathscr{K{ \cup \mathscr{B}}})$. Then by induction, $o^+(G^L) = o^-(G^L)$, so she wins $G$ going first under mis\`ere play with the same move. If Left wins $G$ going first under normal play by moving to $G^L \in \ac( (\mathscr{A}\setminus \mathscr{K} )\cup \mathscr{B})$, then it must be that $G^L = A^L+B$ for some $A^L \in \mathscr{A}\setminus \mathscr{K} $. Since $A^{\mathcal{L}} \cap (\mathscr{A} \setminus \mathscr{K})$ is star-closed, there exists a left option ${A^L}' \in \mathscr{A}\setminus\mathscr{K} $ such that ${A^L}' = A^L+*$ in normal play. Then we calculate that $o^-({A^L}' +B) = o^+({A^L}' + B + *) =o^+(A^L+*+B+*) = o^+(A^L+B)$, so Left wins going first under mis\`ere play by moving to ${A^L}'$.

The proof that Left wins $G$ going first under normal play if she wins $G$ going first under mis\`ere play is symmetric. If Left wins $G$ going first under mis\`ere play by moving to $G^L \in\ac(  \mathscr{K} \cup \mathscr{B})$, then just as before,  she wins $G$ going first under normal play with the same move. If Left wins $G$ going first under mis\`ere play by moving to $G^L  = A^L + B \in \ac( (\mathscr{A}\setminus \mathscr{K} )\cup \mathscr{B})$, then we calculate that she wins $G$ going first under normal play by moving to ${A^L}' = A^L +*$, since $o^+({A^L}' +B) = o^+(A^L+*+B) = o^-(A^L+*+B+*) = o^-(A^L+B)$. 

    Say that $A \in \mathscr{A} \setminus \mathscr{K}$. Then we need to show that $o^+(G) = o^-(G+*)$ and $o^-(G) = o^+(G+*)$. 

    We first show $o^+(G) = o^-(G+*)$. Say that Left wins $G$ moving first under normal play. If she wins by moving to $G^L \in \ac( (\mathscr{A}\setminus \mathscr{K} )\cup \mathscr{B})$, then she wins $G+*$ under mis\`ere play by moving to $G^L$ as well, since $o^+(G^L) = o^-(G^L+*)$. If Left wins $G$ moving first under normal play by moving to $G^L \in \ac( (\mathscr{A}\setminus \mathscr{K} )\cup \mathscr{B})$, then she either moved to $A^L + B$ with $A^L \in \mathscr{K}$, or $A+B^L$ with $B^L \in \ac( (\mathscr{A}\setminus \mathscr{K} )\cup \mathscr{B})$. But $A^{\mathcal{L}} \cap  \mathscr{K}$ and $B^{\mathcal{L}}  \cap \text{ac}(\mathscr{K} \cup \mathscr{B})$ are both star-closed, so either way there exists ${G^L}'$ such that ${G^L}' = G^L+*$ in normal play. Then $o^+({G^L}') = o^+(G^L+*) = o^-(G^L+*+*) = o^-(G^L)$, where the second equality is justified by the fact that we are inducting on the birthday of the normal play canonical form of $G$. Since $G^L+* = G^{L'}$, which is an option of $G$, the birthday of the normal play canonical form of $G^L +*$ is less than that of $G$, proving the induction is justified.

The other half of the proof is nearly symmetric. Say Left wins $G+*$ going first under mis\`ere play. By assumption, if $G \in \ac((\mathscr{A}\setminus\mathscr{K} )\cup 
\mathscr{B})$, then moving to $G$ is never the unique winning move in $G+*$ under either play convention for either player. So Left may win $G+*$ going first under mis\`ere play by moving to some $G^L+*$. If $G^L \in \ac( (\mathscr{A}\setminus \mathscr{K} )\cup \mathscr{B})$, then $o^-(G^L+*) =o^+(G^L)$, so Left wins $G$ going first in normal play by moving to $G^L$. If $G^L \in \ac(  \mathscr{K} \cup \mathscr{B})$, then there exists ${G^L}' \in\ac(  \mathscr{K} \cup \mathscr{B})$ such that ${G^L}' = G^L+*$. Then $o^+({G^L}') =o^+(G^L+*) = o^-(G^L+*)$, so Left wins $G$ going first under normal play by moving to ${G^L}'$.

The proof that $o^-(G) = o^+(G+*)$ is symmetric.
\end{proof}
\section{Wildflowers}
\label{wildflowers}
We now apply the theorems of the last section to find larger sets with the evil twin property.
\begin{defn}
A \emph{star-based} game is any game of form $*:H$.
\end{defn}
Both \cite{mckay2016misere} and \cite{lomisere} ask what other sets of star based games there are with the evil twin property. Note that not all games with the evil twin property are star based, as shown in Example \ref{star based counterexample}.

We will find a set of games with the evil twin property that includes many star based games. The games we consider will all be of the following form. 
\begin{defn}
    A \emph{wildflower} is a game of form $G:H$, where $G$ is a impartial game and $H$ can be any game. 
\end{defn}
\begin{defn}
    A wildflower $G:H$ is \emph{tame} if $G$ is tame.
\end{defn}
\begin{defn}
    A tame wildflower $G:H$ is said to be \emph{fickle} if $G$ is fickle, and \emph{firm} if $G$ is firm.
\end{defn}

The fickleness or firmness of a tame wildflower can depend on how it is written. For example, $*2:1$ is firm, but it is literally equal to $*:(*:1)$, which is fickle. This will not affect our results, so we will ignore this issue.

Our main theorem will be that $\cl(\mathscr{S} \cup \mathscr{H})$ has the evil twin property with evil kernel $\cl(\mathscr{S}\cup \mathscr{H}) \setminus \ac(\mathscr{S})$, where $\mathscr{H}$ and $\mathscr{S}$ are subsets of the firm and fickle wildflowers, respectively, called the \emph{restricted} fickle wildflowers and \emph{restricted} firm wildflowers, which will be defined later in the section.

A subset of wildflowers we are particularly interested in are the mutant flowers. \begin{defn}
    A \emph{mutant wildflower} is a game of form $\{*n_1,\dots,*n_m\}:a$, where $a$ is a dyadic rational number.
\end{defn}
This is different from $*n:a$ because the ordinal sum depends on the literal form of the game, not the game up to equality. As a result of our main theorem, we will find the largest set of mutant wildflowers with the evil twin property.

\begin{defn}
    A fickle game $G$ is \emph{restricted fickle} if the following property holds: for all fickle subpositions $G'$ of $G$, the set of Left options $G'^{\mathcal{L}}$ is star closed, and the set of Right options $G'^{\mathcal{R}}$ is star closed.
\end{defn}

\begin{defn}
    $\mathscr{R}_i, i \in \{0,1\}$  is defined as the set of all short games $G$ with the following properties:

    \begin{enumerate}
        \item $G \not \in \mathcal{N}^+$.
        \item Let $G^{\mathcal{L}} \cap \mathcal{N}^+ =\{ G_1 ,\dots G_n \}$. Then $\{ *\mathcal{G}^+(G_1+i), \dots,* \mathcal{G}^+(G_n + i) \}$ is star closed. 
        \item Let $G^{\mathcal{R}} \cap \mathcal{N}^+ =\{ G_1 ,\dots G_n \}$. Then $\{ *\mathcal{G}^+(G_1+i), \dots,*\mathcal{G}^+(G_n + i) \}$ is star closed. 
    \end{enumerate}
    \end{defn}
    
\begin{defn}
    A \emph{restricted fickle} wildflower is a game of form $G: H$, where $G$ is restricted fickle and $H \in \mathscr{R}_{\mathcal{G}^+(G)}$.
\end{defn}

Notice that $\mathcal{G}^+(G) \in \{0,1\}$, if $G:H$ is a restricted fickle wildflower, since $G$ is fickle.
These restrictions are necessary to allow us to apply Theorem \ref{add to a- (a-k)} on the set of all restricted fickle games.
We need a similar definition for restricted firm wildflowers, so that the set of sums of restricted fickle and firm wildflowers are closed.

\begin{defn}
    A firm game $G$ is \emph{restricted firm} if every fickle subposition $G'$ of $G$ is restricted fickle.
\end{defn}

\begin{defn}
     A \emph{restricted firm wildflower} is a game of form $G: H$, where $G$ is restricted firm wildflower and $H$ can be any game.
\end{defn}

Now that we have defined all the terms of Theorem \ref{main theorem 1, minus teminology}, we can restate it in terms of the notation of Section \ref{3}.
\begin{thm}
\label{big}
 Let $\mathscr{H}$ be the set of all restricted firm wildflowers, and $\mathscr{S}$ be the set of all restricted fickle wildflowers. Then $(\text{cl}(\mathscr{S} \cup \mathscr{H}),\text{cl}(\mathscr{S} \cup \mathscr{H}) \setminus \cl(\mathscr{S}))$ is evilly normal.
\end{thm}
Notice that $\text{cl}(\mathscr{S} \cup \mathscr{H})$ does \emph{not} contain all the tame games, but only the restricted firm and fickle games, even though the set of all tame games has the evil twin property.

Theorem \ref{big} can explain why $(\mathscr{A},\mathscr{K})$ is evilly normal even when $\mathscr{A}$ is not a subset of $\text{cl}(\mathscr{S} \cup \mathscr{H})$. For example, we have shown that $(\cl (\{0,* \mid *\} , \{*\mid 0,*\}, *2:1, *2:-1 ,), \cl (\{0,* \mid *\} , \{*\mid 0,*\}, *2:1, *2:(-1),) \setminus \cl (\{0,* \mid *\} , \{*\mid 0,*\}))$ is evilly normal. Although $\{ *\mid 0,*\}$ is not literally equivalent a tame wildflower, the equivalence $*+* \equiv_{\mathcal{D}} 0$ implies that $\{*\mid0,*\} \equiv_{\mathcal{D}}\{*\mid*+*,*\} \equiv (*+*):(-1) $ and $\{*\mid0,*\} \equiv_{\mathcal{D}} (*+*):1$. Since $*+*$ is a restricted fickle game, $\{*\mid 0,*\}$ and $\{0,*\mid*\}$ behave like restricted fickle wildflowers, which proves that the set is evilly normal.

    \section{Results on Wildflowers}
\label{5}
Now we begin to prove the main theorem of our paper. This will require several theorems specific to sums of wildflowers.
We begin with Theorem 2.6 from \cite{mckay2016misere}. 

\begin{thm}
  \label{ord_sum_ord}
    Let $G$ have a left and a right option. Then 
    \begin{enumerate}
        \item   If $H \in \mathcal{L}^+ \cup \mathcal{P}^+$, then $G:H\geq^- G$.
        \item  If $H \in \mathcal{R}^+ \cup \mathcal{P}^+$, then $G:H \leq^- G$.
    \end{enumerate}
\end{thm}

  \begin{proof}
      Up to symmetry, it is only necessary to show the first part. Assume that Right loses moving first in $H$, that is, $H \in \mathcal{L}^+ \cup \mathcal{P}^+$. We only need to show that $G:H \geq^- G$. Let $X$ be any game. We will show that $o^-(G:H+X) \geq o^-(G+X)$.

      If Left wins going first or second in $G+X$, she can win in $G:H+X$ with the following strategy: If Right plays in $H$, play according to her winning strategy in $H$, else play according to the original strategy for $G+X$. Since $G$ has Left and Right moves, a move in $G:H$ must be made eventually, and then Left will be left with a position she could have also gotten following her original strategy in $G+X$.
\end{proof}
We apply this to wildflowers.
\begin{defn}
    A wildflower $G:H$ is said to be \emph{blue} if $H \in \mathcal{L}^+$, and \emph{red} if $H \in \mathcal{R}^+$. Blue and red flowers are collectively referred to as \emph{colorful} wildflowers.
\end{defn}
Theorem \ref{ord_sum_ord} tells us that $L \cong G:H \geq^- *n$ if $L$ is a blue wildflower, and $R\cong G:H \leq^- G$ if $R$ is a red wildflower, and $G:H =^- *n$ if $H \in \mathcal{P}^+$. These are all the outcome classes of games in $\mathscr{R}_i$. These properties are enough to prove a generalized mis\`ere version of the blue flower ploy proven on page 199 of \cite{WinningWays}. 

\begin{thm}[Blue Flower Ploy]
  \label{atomic_weight_og}
    Let $G $ be a sum of flowers and nimbers. Let the number of blue flowers be $\ell$, and the number of red flowers be $r$. Then,
\begin{enumerate}
    \item  if $\ell-r \geq 2$, then $G \in \mathcal{L}^+$,
    \item if $\ell-r \geq 1$, then $G \in \mathcal{L}^+ \cup\mathcal{N}^+$, 
    \item  if $\ell-r \leq -1$, then $G \in \mathcal{R}^+ \cup\mathcal{N}^+$, 
    \item if $\ell-r \leq-2$, then ${G}\in\mathcal{R}^+$.
\end{enumerate}
\end{thm}

The proof relies on the fact that the atomic weights of a generalized flower are $\pm 1$. Using atomic weights, we can generalize this theorem to all all-small games.
\begin{thm}
\label{2aheadPloy}
    Let $G$ be an all-small game. Then, 
    \begin{enumerate}
    \item  if $\aw(G)\geq 2$, then $G \in \mathcal{L}^+$,
    \item if $\aw(G) \geq 1$, then $G \in \mathcal{L}^+ \cup\mathcal{N}^+$, 
    \item  if $\aw(G)\leq -1$, then $G \in \mathcal{R}^+ \cup\mathcal{N}^+$, 
    \item if $\aw(G)\leq-2$, then ${G}\in\mathcal{R}^+$.
\end{enumerate}
\end{thm}

This is Theorem 98 of \cite{conway2000numbers}.
Theorem \ref{atomic_weight_og} translates exactly to sums of colorful wildflowers and impartial games under mis\`ere play.

\begin{thm}
  \label{atomic_weight}
    Let $G $ be a sum of colored wildflowers and impartial games. Let the number of blue wildflowers be $\ell$, and the number of red wildflowers be $r$. Then 
\begin{enumerate}
    \item  if $\ell-r \geq 2$, then $ G \in \mathcal{L}^-$,
    \item if $\ell-r \geq 1$, then $G \in \mathcal{L}^- \cup\mathcal{N}^-$ ,
    \item  if $\ell-r \leq -1$, then $G \in \mathcal{R}^- \cup\mathcal{N}^-$,
    \item if $\ell-r \leq -2$, then $ {G}\in\mathcal{R}^-$.
\end{enumerate}
\end{thm}

\begin{proof}
   A sum of blue wildflowers, red wildflowers, and impartial games is of form \begin{equation}
       G \cong \sum_{i = 1}^{\ell} G_i:L_i + \sum_{j = 1}^r G'_j:R_j +N
   \end{equation}
   where $L_i \in \mathcal{L}^+,R_j \in \mathcal{R}^+$ and $G_i,G'_j,N$ are all impartial. The first sum corresponds to the blue wildflowers, the second to red wildflowers, and the third to the impartial games.

Up to symmetry, we only need to prove the first two parts of the Theorem.

 We first prove the $\ell = 1,r=0$ case of the theorem. We need to show that $G \cong G_1:L_1 + N \in \mathcal{L}^- \cup \mathcal{N}^-$, where $L_1 \in \mathcal{L}^+$. If $G_1 + N\in \mathcal{N}^-$, then since $G \geq^- G_1 + N$ by Theorem \ref{ord_sum_ord}, Left wins $G$ going first in mis\'ere play as well. If $G_1 + N\in \mathcal{P}^-$, Left wins $G$ going first by making a move in $L_1$ to $L_1^L$, where $L_1^L \in \mathcal{L}^+\cup \mathcal{P}^+$, since \ref{ord_sum_ord} shows that $G_1:L_1^L + N\geq^-  G_1+ N  \in \mathcal{P}^-$, meaning Right loses moving first in $G_1:L_1^L +N $. Thus, Left wins going first under mis\`ere play, so $G \in \mathcal{L}^- \cup \mathcal{N}^-$. 
 
   Now we consider the case $\ell \geq 2,r=0$. We need to show that $G \cong  \sum_{i = 1}^\ell G_i:L_i + N \in \mathcal{L}^-$ for $\ell \geq 2$, where $L_i \in \mathcal{L}^+$. It is clear that $G \in \mathcal{L}^- \cup \mathcal{N}^-$, since by Theorem \ref{ord_sum_ord}, $G \geq^- G_1:L_1 + \sum_{i=2}^\ell G_i +N \in \mathcal{L}^- \cup \mathcal{N}^-$ by the last paragraph. Thus any sum of blue wildflowers and nimbers is in $\mathcal{L}^- \cup \mathcal{N}^-$. But $G$ has more than one blue wildflower, so all of Right's moves in $G$ leave at least one blue wildflower unharmed, and Left wins going first in the resulting position. Thus Right loses $G$ going first under mis\`ere play, so $G \in \mathcal{L}^-$.

Now we are ready to work on the general case. We will induct on the total number of flowers, $\ell + r$. Say that the initial position $G$ has $\ell$ blue flowers and $r$ red flowers, with $\ell \leq r$. If Left is to move, moving in a red flower makes the difference between the number of blue flowers and red flowers at least $2$, so Left has won by moving to an $\mathcal{L}^-$ position. If Right is to move, the inductive hypothesis shows that he loses if he moves in a red flower or nimber, so he must move in a blue flower and reduce their number by one. Then Left can respond by removing a red flower. Then the difference between the number of blue flowers and red flowers becomes $\ell - 1 - (r -1)$, so we are reduced to the $\ell-1,r-1$ case, which we can assume has already been proven by induction.
\end{proof}
Before finding a set of wildflowers with the evil twin property, we need to prove the final requirement of Theorem \ref{add to a- (a-k)} holds.
\begin{lemma}
    \label{star moves suboptimal}
    Let \begin{equation}
        G  \cong \sum_{i=1}^{\ell} G_i:L_i+\sum_{j=1}^{r} G_j':R_j  +N
    \end{equation}
    where $L_i \in \mathcal{L}^+,R_i \in \mathcal{R}^+$ and $N$ is fickle, be a sum of restricted fickle wildflowers and a fickle game. Then moving from $G+*$ to $G$ is never the unique winning move in either play convention unless $G \cong 0$.
\end{lemma}
\begin{proof}

If $\ell>0$ and $r>0$, this is simple, since Theorems \ref{atomic_weight_og} and \ref{atomic_weight} prove that it is best for both sides to move in a wildflower of the other sides color. So assume there are only wildflowers of one color, say blue.
 
If there are no blue wildflowers, we are just playing with sums of fickle games, which is easy under both play conventions.

First suppose there is only one blue wildflower, so that $G+* \cong G_1:L_1 + N+*$ for some restricted fickle $G_1,N$. Say that $G_1 + N+* \in \mathcal{P}^-$. Then Left wins by moving to $G_1:L_1^L + N+*$, where $L_1^L \in \mathcal{L}^+\cup\mathcal{P}^+$. So assume that $G_1 + N +* \in \mathcal{N}^-$. Then Left wins by moving to any $G_1^L + N+*$, since $ 1= \mathcal{G}(G_1 +N+*) \neq \mathcal{G}(G_1^L+N+*)$ and $G_1 + N+*$ is a fickle game. 

If we have more than one blue wildflower, Theorem \ref{ord_sum_ord} shows that $G+* \geq G_1:L_1 + N+*$ for some restricted fickle $G_1,N$, and we can proceed like before
\end{proof}
We also recall the colon principle for normal play impartial games \cite[191]{WinningWays}. 

\begin{thm}[Colon Principle]
    Let $G$ and $H$ be impartial games. Then $\mathcal{G}^+(G:H) = \mathcal{G}^+(G) + \mathcal{G}^+(H)$.
\end{thm}
The proof of the main Theorem now comes down to just checking the conditions to apply Theorems \ref{a-k extension} and \ref{add to a- (a-k)}.

\begin{proof}[Proof of Theorem \ref{big}]
Let $\mathscr{A}$ be the set of all impartial restricted tame games, and $ \mathscr{K}$ be the set of all impartial restricted firm games, so that $\mathscr{A} \setminus \mathscr{K}$ is the set of all impartial restricted fickle games.

Let $\mathscr{H}$ be the set of restricted firm wildflowers. Since every restricted firm game has a move to restricted fickle games equal in normal play to $0$ and $*$, the set of moves to restricted fickle games is star closed. Applying Theorem \ref{a-k extension} with $\mathscr{B} = \mathscr{H}$ shows that $(\cl(\mathscr{A} \cup \mathscr{H}),\cl(\mathscr{A} \cup \mathcal{H}) \setminus (\mathscr{A} \setminus \mathscr{K} )$ is evilly normal.

Let $\mathscr{S}$ be the set of restricted fickle wildflowers. Now, we can apply Theorem \ref{add to a- (a-k)} with $\mathscr{B} = \mathscr{S}$. Lemma \ref{star moves suboptimal} shows that for any $G \in \ac (\mathscr{S})$, it is never worthwhile to move from $G+*$ to $G$. The definition of $\mathscr{R}_i$ shows that $ \mathscr{S} \cup \cl(\mathscr{A} \cup \mathscr{H})  $ is hereditarily closed. It remains to prove that the set of moves to the kernel $\mathscr{H}$ is star closed. Let $G:H$ be a restricted fickle wildflower, with $G$ a restricted fickle game and $H \in \mathscr{R}_{\mathcal{G}^+(G)}$. The only moves from $G:H$ to restricted firm positions are those to $G:H'$, where $H' \in \mathcal{N}^+$ , and those to $G'$, a firm option of $G$. 

Let the restricted firm options of $G$ be $G_1',G_2',\dots,G_n'$. Let the set of options of $H$ with outcome class $\mathcal{N}^+$ be $H_1' ,\dots H_m'$. Then the set of restricted firm options of $G:H$ is $\{G_1',\dots,G_n',G:H_1',\dots,G:H_m' \}$. Our goal is to prove that this set is star closed.
Since $G$ is restricted fickle, $\{G_1',G_2',\dots,G_n'\}$ is a star closed set.
And since $H \in \mathscr{R}_{\mathcal{G}^+(G)}$, the set $\{ *(\mathcal{G}^+(H_1')+\mathcal{G}^+(G)), \dots, *(\mathcal{G}^+(H_m')+\mathcal{G}^+(G) )\}$ is star closed by the definition of $\mathscr{R}_{\mathcal{G}^+(G)}$. But by the colon principle, $\mathcal{G}^+(G:H_i') = \mathcal{G}^+(G) + \mathcal{G}(H_i')$, so $G:H_i'$ has the same Grundy value as $*(\mathcal{G}^+(H_1')+\mathcal{G}^+(G))$, and $\{G:H_1', \dots,G:H_m'\}$ is star closed. Since both $\{G:H_1', \dots,G:H_m'\}$ and  $\{G_1,G_2,\dots,G_n\}$ are star closed, their union is as well, proving the last hypothesis of Theorem \ref{add to a- (a-k)}.

Thus $(\cl(\mathscr{S} \cup \mathscr{T}),\cl(\mathscr{S} \cup \mathscr{T}) \setminus \ac(\mathscr{S}))$ is evilly normal.
\end{proof}

Now that we have finally proven Theorem \ref{big}, we can apply it to mutant flowers. Recall the definition of mutant flowers. \begin{defn}
    A mutant flower is a game of form $\{*x_1,\dots,*x_n\}:a$, where the $x_i$ are integers and $a$ is a dyadic rational. The height of $\{*x_1,\dots,*x_n\}:a$ is said to be  $\mex(x_i)$.
\end{defn}
\begin{defn}
    The set of all mutant flowers is denoted $\mathscr{M}$. 
\end{defn}
The set $\mathscr{S} \cup \mathscr{H}$ of Theorem \ref{big} includes many mutant flowers.
 Specifically, $\{*x_1,\dots,*x_n\}:a \in \mathscr{S}$ if and only if $\text{mex}({x_1}) >1$, and $\{*x_1,\dots,*x_n\}:a \in \mathscr{S}$ if $\mex(x_i) \leq 1$ and $\{*x_1,\dots,*x_n\} \setminus \{0,1\}$ is star-closed.
A bit of calculation actually shows that this is the maximal set of mutant flowers with the evil twin property.\begin{thm}
  Let $\text{mex}(\{x_i\})\leq1$. Then
    \begin{equation}
        \{*x_1,\dots ,*x_m\}:1 + \{*x_1,\dotsm*x_m\}:-1 \in  \mathcal{P}^+ \cap \mathcal{N}^-
    \end{equation}
    and \begin{equation}
     \{*x_1,\dots ,*x_m\}:1 + \{*x_1,\dotsm*x_m\}:-1 + *\in  \mathcal{N}^+,
    \end{equation}
    but \begin{equation}
          \{*x_1,\dots ,*x_m\}:1 + \{*x_1,\dotsm*x_m\}:-1+*\in  \begin{cases}
              \mathcal{P}^- & \text{if } \{x_i\} \setminus \{0,1\} \text{ is star-closed}\\
              \mathcal{N}^-& \text{if } \{x_i\} \setminus \{0,1\} \text{ is not star-closed}. \\
          \end{cases}
    \end{equation}
\end{thm}
\begin{proof}
We will only consider the $\text{mex}(\{x_i\})=1$ since the $\text{mex}(\{x_i\})=0$ case follows similarly.
  
  The games are all symmetric, so we only need to find if Left wins going first. The normal play case is simple, since $\{*x_1\dots,*x_m\}:1 + \{*x_1\dots,*x_m\}:-1 = 0$. So $\{*x_1\dots,*x_m\}:1 + \{*x_1\dots,*x_m\}:-1 = 0 \in \mathcal{P}^+$ and  $\{*x_1\dots,*x_m\}:1 + \{*x_1\dots,*x_m\}:-1 + * = * \in \mathcal{N}^+$. In the mis\`ere play case, Theorem\ref{atomic_weight} shows that, Left going first must move in $\{*x_1\dots,*x_m\}:-1$, followed by Right moving in $\{*x_1\dots,*x_m\}:1$, after which we get a nim position.
  
In mis\`ere play  Left wins $ \{*x_1\dots,*x_m\}:1 + \{*x_1\dots,*x_m\}:-1$ going first by moving to $0$, since Right cannot move to $*$ in $ \{*x_1\dots,*x_n\}:1$.

  Now consider $G =\{*x_1\dots,*x_m\}:1 + \{*x_1\dots,*x_m\}:-1 +*$ in mis\`ere play. Left loses if she moves to $0$, since Right can also move to $0$. So she must move to some $*x_i$, with $x_i$ positive. Right must move to a nimber $*x_j$ from $ \{*x_1\dots,*x_n\}:1$. If $*x_j \cong *(x_i \oplus 1)$, Right wins, and else he loses. So, if $\{*x_1,\dots,*x_m\}$ is star-closed, Right always has a winning move to $*(x_i\oplus 1)$. So Left wins going first, and by symmetry $G \in \mathcal{P}^-$. If $\{*x_1,\dots,*x_m\}$ is not star-closed, then there exists some $*x_i$ for which $*(x_i \oplus 1) \notin \{*x_1 , \dots, *x_m \}$, so Right loses if Left moves to $*x_i$. By symmetry, $G \in \mathcal{N}^-$.
\end{proof}
Thus we have shown the following theorem \ref{main theorem for mutant flowers, minus terminology}, restated below with our new terminology.

\begin{thm}
    Let $\mathscr{M}$ be the set of all sums of mutant flowers. The maximal set of mutant flowers with the evil twin property is $\cl(\mathscr{S} \cup \mathscr{T}) \cap \mathscr{M}$, which has evil kernel $(\cl(\mathscr{S} \cup \mathscr{T}) \setminus \ac(\mathscr{S}) ) \cap \mathscr{M}$.
\end{thm}

\section{Hardness of Restricted Wildflowers}
\label{6}

We have proven that a large set of wildflowers have the evil twin property, which tells us how to find mis\`ere-play outcome classes from normal play outcome classes. But it remains to see how hard it is to find the normal play outcome class of a sum of wildflowers. It is shown in \cite{mckay2016misere} that the optimal move under mis\`ere play for Left in a sum of sprigs (games of form $*:a$) under mis\`ere play is to move to $0$ from $*:a$, where $a$ is minimal, and vice versa for Right.

Does Theorem \ref{big} allow us to find mis\`ere play outcome classes of sums of a larger set of wildflowers? The set $\text{cl}(\mathscr{S} \cup \mathscr{T})$ seems too large for a simple strategy to exist. But it seems plausible that there is a fast way to find the outcome classes of sums of mutant flowers in $\cl(\mathscr{S} \cup \mathscr{T})$.

On the contrary, we will show that this problem is $\mathsf{NP\text{-}hard}$. 

\begin{thm}
    \label{hard}
    Determining the outcome class of a sum of $n$ mutant flowers is $\mathsf{NP\text{-}hard}$, even if we restrict to sums of flowers of the form $\{*x_1,\dots,*x_n\}:\pm 1$, where $\pm 1$ denotes \emph{either} $1$ or $-1$.
\end{thm}

We call the mutant flowers of form $\{*x_1 ,\dots,*x_n\}: \pm 1$ \emph{short mutant flowers}. A related problem was investigated in \cite{burke2024tractability}. These mutant flowers are all superstars according to the definition of \cite{burke2024tractability}.

\begin{defn}
    A superstar is a game whose options are all nimbers.
\end{defn}

In \cite{burke2024tractability} Burke, Ferland, Huntemann, and Teng prove that determining the outcome class of a sum of superstars is $\mathsf{NP\text{-}hard}$ by reducing to a new ruleset \textsc{EPMX}, then proving that \textsc{EPMX} is $\mathsf{NP\text{-}hard}$ by reducing to \textsc{3-Sat}. However, the superstars used in their reduction are not mutant flowers, so we cannot just cite their proof.

Inspired by this reduction, we will reduce from a sum of mutant flowers to 3\textsc{-SAT} directly, without going through \textsc{EPMX}. We use the convention that an instance of 3\textsc{-SAT} has clauses of length $2$ or $3$.

It is shown in \cite{tovey3SAT} that \textsc{3-Sat} is $\mathsf{NP\text{-}complete}$ even if we add the two assumptions that there are an odd number of variables and that every variable $x_i$ appears as a literal in the formula exactly three times: once unnegated as $x_i$, and twice negated as $\neg x_i$. So we will restrict to \textsc{3-Sat} positions with these properties.

Given any \textsc{3-Sat} position $\Pi$, we will find a game $G_{\Pi}$ such that $\Pi$ is satisfiable if and only if $G_{\Pi} \in \mathcal{L }^+ \cup \mathcal{ N}^+$, that is, if Left wins $G_{\Pi}$ moving second under normal play. From now on, all games will be assumed to be played under normal play. The proof that this reduction works will be presented as a series of lemmas.
\begin{ex}
    Throughout the proof, we will use the \textsc{3-SAT} instance $\Omega = (x_1 \vee x_2 ) \land (\neg x_2 \vee \neg x_3) \wedge (\neg x_1 \vee \neg x_3) \wedge ( x_1 \vee\neg x_2 \vee x_3 )$ as an example with the necessary properties. $\Omega$ can be satisfied by assigning $x_1$ to be true, and $x_2$ and $x_3$ to be false.
\end{ex}
\begin{defn}
     Let $\Pi$ be a \textsc{3-SAT} position with the above properties. Then $G_{\Pi}$ is defined as follows. Number the clauses of $\Pi$ from $1$ to $m$. Say that the variable $x_i$ appears unnegated in the clause $r_i$, and negated in $s_i,t_i$, with $s_i < t_i$. 
     
     We will define $n$ red wildflowers $X_i$ corresponding to each variable $x_i$. Each of the letters we define will become an option of $X_i$. Let $a_i = 2^{r_i},b_i = 2^{s_i}+2^{t_i},c=2^{s_i},d=2^{t_i}$. Then let \begin{equation}
X_i = \{0,*,*a_i,*b_i,*c_i,*d_i,*2^{m+i}\}:-1.
\end{equation}
We also define $n$ blue wildflowers. For $1 \leq i \leq n$, we let

\begin{equation}
    Y_i = *:1.
\end{equation}

Finally, we define \begin{equation}
    G_{\Pi} = X_1 + \cdots + X_n + Y_1 + \cdots + Y_n + *(2^{m+1}-2).
\end{equation}
\end{defn}
\begin{ex}
        In $\Omega$, the clause $x_1 \vee x_2$ is numbered $1$, the clause $\neg x_1 \vee \neg x_3$ is numbered $2$, and so on. Then $x_1$ appears only in clause $1$, so $r_1 = 1$. Since $\neg x_1$ appears in clauses $3$ and $4$, $s_1 = 3$ and $t_1 = 4$. We find the other values of $r_i,t_i$ and $s_i$ similarly.

Thus we have that $a_1 = 2^1 = 2,b_1 =  2^3 + 2^4 = 24,c_1 = 2^3 = 8,d_1 = 2^4 = 16,2^{m+i}=32$, so \begin{equation}
    X_1= \{ 0,*,*2,*24,*8,*16,*32 \}:-1.
\end{equation} 
Similarly, $X_2,X_3$ can be found to be \begin{equation}
    X_2 = \{ 0,*,*2,*20,*4,*16,*64\}:-1
\end{equation}
and \begin{equation}
    X_3 =\{ 0,*,*16,*12,*4,*8,*128 \}:-1.
\end{equation}
Since we know $X_1,X_2,X_3$, we can find $G_{\Omega}$.
\end{ex}
\begin{lemma}
    Let $G = \{*x_1,\ldots *x_n\}:1$ be a mutant wildflower of positive height. Then $\aw(G) = 1$. Symmetrically, if $G = \{*x_1,\ldots *x_n\}:-1$, then $\aw(G) = -1$.
\end{lemma}

\begin{proof}
 Up to symmetry, we can assume $G = \{*x_1,\ldots *x_n\}:1$. Proving that $\aw(G) = 1$ is equivalent to showing that \begin{equation}
     *N + \mathord{\downarrow}  < (\{*x_1,\ldots *x_n\}:1 )- \uparrow \mathord  <*N +  \uparrow,
 \end{equation}where $N$ is sufficiently large. 
 
 Proving $*N + \mathord{\downarrow} < (\{*x_1,\dots,*x_n\}:1) -\mathord{\uparrow}$ is equivalent to showing $0 <( \{*x_1,\dots,*x_n\}:1)+*N$. If Left goes first in normal play, we claim she wins by moving in $*N$ to $*\mex ({x_i})$. If Right responds by moving in $*\text{mex} ({x_i})$ to $*y $, Left moves in $\{*x_1,\ldots *x_n\}:1$ to $*y$ as well, leaving Right to move in $*y+*y  = 0\in \mathcal{P}^+$. If instead Right moves in $\{*x_1,\ldots *x_n\}:1$ to $*x_i$, then Left wins $*\mex (x_i) + *x_i $, since $*y \notin \{*x_1,\dots,*x_n\}$. Now say Right goes first. If he moves in $\{*x_1,\ldots *x_n\}:1$ to $*x_i$, then Left wins by moving in $*N$ to $*x_i$, leaving Right to move in $*x_i + *x_i = 0 \in \mathcal{P}^+$. Say he moves in $*N$ to $*N'$. If $N' > *\mex(x_i) $, Left wins by moving in $*N'$ to $*\mex(x_i)$ as before. If $N' > \mex (x_i)$, then $*N' \in \{*x_1 ,\dots,*x_n \}$, so Left can move in $\{*x_1,\ldots *x_n\}:1$ to $*N'$, and Right loses $*N' + *N'$.
Thus we have shown that $*N + \mathord{\downarrow} < (\{*x_1,\dots,*x_n\}:1) - \mathord{\uparrow}$.

It remains to show that $ (\{*x_1,\ldots *x_n\}:1 )- \mathord{\uparrow} <*N + \mathord{\uparrow}$, or $0<(\{*x_1,\ldots *x_n\}:-1 ) + \mathord{\Uparrow} +*N$. If Left is first to move, then she wins by moving in $\{*x_1,\ldots *x_n\}:-1 $ to $0$ since the position has atomic weight $2$. If Right moves in $\{*x_1,\ldots *x_n\}:-1$, he loses because $*i + \mathord{\Uparrow} \in \mathcal{L} + *N $ for all $i$. If instead he moves in $\mathord{\uparrow}$ to $*$, the only Right option, then Left moves in $(\{*x_1,\ldots *x_n\}:-1 )$ to $0$, leaving $\mathord{\uparrow} + *+ *N $, which Right loses.  Finally, if he moves in $*N$ to $*N'$, then Left wins by moving in $(\{*x_1,\ldots *x_n\}:-1 )$ to $0$ as before. Thus $(\{*x_1,\dots,*x_n\}:1) - \mathord{\uparrow} < *N + \mathord{\uparrow}$, and $\aw(G) = -1$.
\end{proof}
This implies that the atomic weight of a sum of red and blue mutant wildflowers and nimbers is just the difference between the number of blue and red wildflowers. So we get the following lemma.
\begin{lemma}
    In $G_{\Pi}$, if Left can win, she can win by only moving in the red mutant flowers $X_i$ for as long as there are wildflowers remaining, and if Right can win, he can win by only moving in the blue mutant flowers $Y_i$ for as long as there are wildflowers remaining. 
\end{lemma}
\begin{proof}
     Say that the players have alternated moves in $G_{\Pi}$, leaving a subposition $G_{\Pi}'$. Assume that both players have been moving in wildflowers of the other's color to get to $G_{\Pi}$. $G_{\Pi}'$ will be a sum of wildflowers and a nimber in normal play.

     Since Left was second to move in $G_{\Pi}$, if Left is next to move in $G_{\Pi}'$, the number of red wildflowers remaining is one more than the number of blue wildflowers, meaning $G_{\Pi}'$ has atomic weight $-1$. If Left moves in $G_{\Pi}'$ in a blue wildflower $*:1$ or a nimber, she moves to a position with atomic weight $\leq -1$. Right is to move in this position, so he wins by Theorem \ref{2aheadPloy}. Thus, Left must move in a red wildflower $X_i$.

If Right is next to move in $G_{\Pi}'$ then the number of red and blue wildflowers remaining is the same, meaning $G_{\Pi}'$ has atomic weight $0$. If Right moves in $G_{\Pi}'$ in a red wildflower $X_i$, Left is left to move in a position with atomic weight $1$, which she wins by Theorem \ref{2aheadPloy}. Say Right moves in one of the nimbers. We claim Left can win by moving in any $X_j$ to $*2^{m+j}$, and taking the remaining $X_i$s to $0$ after. Since Left keeps moving in the red wildflowers $X_i$, the situation is reversed from before and Right must continue moving in the blue wildflowers $*:1$ to $0$. Then eventually Left will be left to move in a position that has no wildflowers, and only a nimber. This nimber cannot be $0$, since Left moved to $*2^{m+i}$, and no other nimber option has a $1$ in the $(m+i+1)^{\text{st}}$ bit, so Left wins. Thus, Right must always move in $*:1$ to $0$, which is his only option.

So both players move in wildflowers of the other's color.
\end{proof}
\begin{lemma}
    Left should never move to one of the big nimbers $*2^{m+i}$ from $Y_i$.
\end{lemma}
\begin{proof}
     Say that Left moves from the $n$ red wildflowers $Y_1,\ldots,Y_n$ to the nimbers $*\ell_1,\ldots,*\ell_n$, respectively. Right must always move from one of the $n$ copies of $*:1$ to $0$. After both players make $n$ moves, Right is left to move in \begin{equation}
    *\ell_1 + \cdots + *\ell_n +*(2^{m+1}-2)= * (\ell_1 \oplus \dots \oplus \ell_n\oplus( 2^{m+1}-2)) = *N.
\end{equation}
Left wins if $*N =0$, and Right wins otherwise. Notice that the order of Left's moves is irrelevant.

If Left moves in $Y_i$ to $*2^{m+i}$, then $\ell_i = *2^{m+i}$. Since none of the other terms in the sum for $*N$ have a $1$ in the $(m+i+1)^{\text{st}}$ bit, $*N \neq 0$, and Left loses. 
So we can assume that $\ell_i \in \{0,a_i,b_i,c_i,d_i\}$.
\end{proof}
\begin{lemma}
    If $\Pi$ is satisfiable, then $G_{\Pi} \in\mathcal{L}^+ \cup \mathcal{N}^+$.
\end{lemma}
\begin{proof}
    Up to reordering the variables $x_1,\dots,x_n$ of $\Pi$, we can assume that $\Pi$ can be satisfied by assigning the variables $x_1,\dots,x_k$ to be true, and $x_{k+1},\ldots,x_n$ to be false.  Because the order of Left's moves doesn't matter, we will have Left's $i^\text{th}$ move be in $X_i$. Recall that the variable $x_i$ appears in $r_i$ unnegated, and in $s_i$ and $t_i$ negated. After Left has made $i$ moves, note that $G_{\Pi}$ is equal to a sum of wildflowers and the nimber $*N_i = *(\ell_1 \oplus \cdots \oplus \ell_i \oplus 2^{m+1}-2)$. Notice that the binary expansion of $*N_0= *(2^{m+1}-2)$ is $m$ bits of $1$, followed by a single $0$. We need to prove that Left's moves $*\ell_1,\dots,*\ell_n$ can be chosen such that $*N_n = *N = * (\ell_1 \oplus \dots \oplus \ell_n \oplus ( 2^{m+1}-2) )= 0$. Equivalently, we need $\ell_1 \oplus \dots \oplus \ell_n = 0$.

    For $1 \leq i \leq k$, let the $i^\text{th}$ move of Left, $*\ell_i$ be to $*a_i$ if the $(r_i+1)^{\text{th}}$ bit of $*N_i$ is a $1$, and $0$ if otherwise.
For $k < i \leq n$, let the $i^\text{th}$ move of Left to be to $*b_i$ if the $(s_i+1)^{\text{th}}$ and $(t_i+1)^{\text{th}}$ bits of $*N_i$  are $1$, $*c_i$ if the $(s_i+1)^\text{th}$ bit of $*N_i$ is $1$ and the $(r_i+1)^{\text{th}}$ is $0$, $*d_i$ if the $(s_i+1)^\text{th}$ bit of $*N_i$ is $0$ and the $(r_i+1)^{\text{th}}$ is $1$, and $0$ if the $(s_i+1)^{\text{th}}$ and $(t_i+1)^{\text{th}}$ bits of $*N_i$  are $0$. 
Notice that if a bit of $N_i$ is $0$ at some point, the bit cannot become a $1$ again in $N_j $ for $N_j,j>i$. Bits are switched to $0$ by the move to $l_i$ if the assignment of $x_i$ satisfies the corresponding clause of $\Pi$. Because all the clauses of $\Pi$ are satisfied when all the variables have been assigned, the $2^{\text{nd}}$ to $(m+1)^{\text{th}}$ bit all become $0$. These are exactly the bits that are $1$ in $ *2^{m+1}-2$, so $*N_n = 0$. Thus, Left wins $G_{\Pi}$ going second. 
\end{proof}
\begin{ex}
    The formula $\Omega$ is $(x_1 \lor x_2 ) \land (\neg x_2 \lor \neg x_3) \land (\neg x_1 \lor \neg x_3) \land (\neg x_1 \lor \neg x_2 \lor x_3 )$. Since $\Omega$ can be satisfied by assigning $x_1$ to be true, and $x_2$ and $x_3$ to be false, we have $k = 1$. 

First we find $\ell_1$. We have $*N_0 = 30$. $x_1$ in $\Omega$ is true, so we are deciding between $\ell_1 = a_1$ and $\ell_1 = 0$.  Since $r_1 = 1$, and the second bit of $N_0$ is a $1$, we let $\ell_1 = a_1 = 2$. So $*N_1 = *(2 \oplus 30) = *28 $. Now we find $\ell_2$, which can be any of $0,b_2,c_2,d_i$ since $x_2$ is assigned to false in $\Omega$. Since $r_2 = 2$ and $s_2 = 4$, and the third and fifth bits of $N_1$ are both $1$, we should choose $\ell_2 = b_2 = 20$. So $N_2 = 2 \oplus 20\oplus30=8$. Now we find $\ell_3$. We see that $r_3 = 2$ and $s_3 = 3$. Since the second bit is zero and the third is one, we choose $\ell_3 = d_2 = 8$. 
    
      Finally, $*N = *N_3 =*(2 \oplus 20 \oplus 8 \oplus 30) = 0 $, so Left wins $G_{\Omega}$ by moving from $X_i$ to $*\ell_i$.
\end{ex}
\begin{lemma}
    If $G_{\Pi} \in\mathcal{L}^+ \cup \mathcal{N}^+$, then $\Pi$ is satisfiable.
\end{lemma}
\begin{proof}
    Assume that Left wins $G_{\Pi}$ going second in normal play. We will prove that $\Pi$ is satisfiable. Say that Left won by making the moves $*\ell_1,\ldots,*\ell_n$ from the mutant flowers $X_1,\dots,X_n$, respectively. Then we need $\ell_1\oplus \cdots \oplus \ell_n= 2^{m+1}-2$. 
We claim that the following procedure gives an assignment of truth values to $x_1,\dots,x_n$ that satisfies the formula $\Pi$: If $\ell_i = a_i$, let $x_i$ be true; else, assign $x_i$ to be false. 
Notice that the $j^{\text{th}}$ clause is satisfied in $\Pi$ by the assignment of a truth value to $x_i$ if the $(j+1)^{th}$ bit of $\ell_i$ is a $1$. 

Since $\ell_1\oplus \cdots \oplus \ell_n= 2^{m+1}-2$, for each bit from the $2^{\text{nd}}$ to the $(m+1)^{\text{th}}$, there exists an $\ell_i$ with a $1$ in that bit. Thus all the clauses of $\Pi$ are satisfied.
\end{proof}
\begin{ex}
In $\Omega= (x_1 \lor x_2 ) \land (\neg x_2 \lor \neg x_3) \land (\neg x_1 \lor \neg x_3) \land (\neg x_1 \lor \neg x_2 \lor x_3 )$, Left can win by moving to $*\ell_1 = 2,*\ell_2 = 20,*\ell_3 = *8$. This leads to the assignment of $x_1$ to true and $x_2$ and $x_3$ to false, since $*a_1 = 2$.
\end{ex}
Combining the previous two lemmas proves Theorem \ref{hard} by reducing from a 3-\textsc{SAT} instance $\Pi$ to a sum of superstars $G_{\Pi}$.

\section{Further Questions}  
\label{7}
We still know next to nothing about mis\`ere games compared to normal play games. Here are some of the questions related to this paper.
\begin{enumerate}
    \item Can genus theory be generalized to some set of partizan mis\`ere play games?
    \item Is Conjecture \ref{all_norm} true?
    \item   Other than restricted wildflowers, what interesting sets of games have the evil twin property? Are there any such sets that are not equivalent to wildflowers?
    \item  Theorem \ref{a-k extension} allows us to add elements to $\mathscr{K}$ and Theorem \ref{add to a- (a-k)} allows us to add elements to $\mathscr{A}\setminus\mathscr{K}$. What other ways are there to extend $(\mathscr{A},\mathscr{K})$?
    \item Let $(\mathscr{A},\mathscr{K})$ be evilly normal. Say that $\mathscr{A} \in \mathcal{D}$, the set of dicotic games. The games invertible modulo $\mathcal{D}$ are characterized by Fisher, Nowakowski, and Santos in \cite{dicoticinvertible}. Let $G \in \mathscr{A}$ be invertible in $\mathcal{D}$. Is it necessarily true that $G \in \mathscr{A} \setminus \mathscr{K}$?
    \item   Is finding the outcome class of a sum of mutant flowers with the evil twin property $\mathsf{PSPACE\text{-}complete}$? How hard is finding the outcome class of \emph{generalized flowers} of form $*n:a$?
\end{enumerate}
\printbibliography 

@article{dicoticinvertible,
 author = {Fisher, Michael and Nowakowski, Richard J. and Pereira dos Santos, Carlos},
 title = {Invertible elements of the dicot mis{\`e}re universe},
 fjournal = {Integers},
 journal = {Integers},
 issn = {1553-1732},
 volume = {22},
 pages = {paper g6, 11},
 year = {2022},
 language = {English},
 url = {math.colgate.edu/~integers/wg6/wg6.pdf},
 zbMATH = {7633013},
 Zbl = {1507.91034}
}

@book{conway2000numbers,
  author = {Conway, J. H.},
 title = {On numbers and games.},
 edition = {2nd ed.},
 isbn = {1-56881-127-6},
 year = {2001},
 publisher = {Natick, MA: A K Peters},
 language = {English},
 zbMATH = {1568491},
 Zbl = {0972.11002}
}

@article{plambeck2008misere,
  author = {Plambeck, Thane E. and Siegel, Aaron N.},
 title = {Mis{\`e}re quotients for impartial games},
 fjournal = {Journal of Combinatorial Theory. Series A},
 journal = {J. Comb. Theory, Ser. A},
 issn = {0097-3165},
 volume = {115},
 number = {4},
 pages = {593--622},
 year = {2008},
 language = {English},
 doi = {10.1016/j.jcta.2007.07.008},
 zbMATH = {5274181},
 Zbl = {1142.91022}
}

@misc{burke2024tractability,
  author = {Burke, Kyle and Ferland, Matthew and Huntemann, Svenja and Teng, Shanghua},
 title = {A {Tractability} {Gap} {Beyond} {Nim}-{Sums}: {It}'s {Hard} to {Tell} {Whether} a {Bunch} of {Superstars} {Are} {Losers}},
 year = {2024},
 howpublished = {Preprint, {arXiv}:2403.04955 [cs.{CC}] (2024)},
 url = {https://arxiv.org/abs/2403.04955},
 arXiv = {arXiv:2403.04955}

}

@article{mckay2016misere,
  author = {McKay, Neil A. and Milley, Rebecca and Nowakowski, Richard J.},
 title = {Mis{\`e}re-play {Hackenbush} sprigs},
 fjournal = {International Journal of Game Theory},
 journal = {Int. J. Game Theory},
 issn = {0020-7276},
 volume = {45},
 number = {3},
 pages = {731--742},
 year = {2016},
 language = {English},
 doi = {10.1007/s00182-015-0484-8},
 zbMATH = {6641454},
 Zbl = {1388.91079}
}

@incollection{allen2015peeking,
    author = {Allen, Meghan R.},
 title = {Peeking at partizan mis{\`e}re quotients},
 booktitle = {Games of no chance 4. Papers of the BIRS workshop on combinatorial game theory, Banff, Canada, January 20--25, 2008},
 isbn = {978-1-107-01103-8},
 pages = {1--12},
 year = {2015},
 publisher = {Cambridge: Cambridge University Press},
 language = {English},
 zbMATH = {6490897},
 Zbl = {1380.91037}
}

@misc{lomisere,
       author = {Lo, Irene Y.},
 title = {Misere {Hackenbush} {Flowers}},
 year = {2012},
 howpublished = {Preprint, {arXiv}:1212.5937 [math.{CO}] (2012)},
 url = {https://arxiv.org/abs/1212.5937},
 arXiv = {arXiv:1212.5937}
}

@article{ottoway_zero,
 author = {Mesdal, G. A. and Ottaway, P.},
 title = {Simplification of partizan games in mis{\`e}re play},
 fjournal = {Integers},
 journal = {Integers},
 issn = {1553-1732},
 volume = {7},
 number = {1},
 pages = {paper g06, 12},
 year = {2007},
 language = {English},
 url = {https://eudml.org/doc/130369},
 zbMATH = {5171562},
 Zbl = {1165.91008}
}

@book{WinningWays,
 author = {Berlekamp, Elwyn R. and Conway, John H. and Guy, Richard K.},
 title = {Winning ways for your mathematical plays. {Vol}. 2.},
 edition = {2nd ed.},
 isbn = {1-56881-142-X},
 year = {2003},
 publisher = {Natick, MA: A K Peters},
 language = {English},
 zbMATH = {1889828},
 Zbl = {1011.00009}
}

@article{tovey3SAT,
 author = {Tovey, Craig A.},
 title = {A simplified {NP}-complete satisfiability problem},
 fjournal = {Discrete Applied Mathematics},
 journal = {Discrete Appl. Math.},
 issn = {0166-218X},
 volume = {8},
 pages = {85--89},
 year = {1984},
 language = {English},
 doi = {10.1016/0166-218X(84)90081-7},
 zbMATH = {3848607},
 Zbl = {0534.68028}
}

\end{document}